\documentclass[12pt,leqno]{amsart}
\usepackage{amsmath,amsfonts,amssymb,amsthm}
\usepackage{graphicx}
\graphicspath{ }
\usepackage{wrapfig}
\usepackage{accents}
\usepackage{caption}
\usepackage{subcaption}
\usepackage{calligra}

\numberwithin{figure}{section}
\newcommand\norm[1]{\left\lVert#1\right\rVert}

\theoremstyle{plain}
\newtheorem{thm}{Theorem}[section]

\newtheorem{prop}[thm]{Proposition}
\newtheorem*{cor}{Corollary}
\theoremstyle{definition}
\newtheorem{defn}{Definition}[section]

\newtheorem{exmp}{Example}[section]

\theoremstyle{remark}

\title{Some results on $\varphi$--convex functions and geodesic $\varphi$-convex functions}
\author[A. A. Shaikh, Akhlad Iqbal and C. K. Mondal]{Absos Ali Shaikh$^1$, Akhlad Iqbal$^2$ and Chandan Kumar Mondal$^3$}
\address{\noindent\newline $^{1,3}$ Department of Mathematics,\newline University of
Burdwan, Golapbag,\newline Burdwan-713104,\newline West Bengal, India}
\email{aask2003@yahoo.co.in, aashaikh@math.buruniv.ac.in}
\email{chan.alge@gmail.com}

\address{\noindent\newline $^{2}$ Department of Mathematics,\newline Aligarh Muslim University,\newline Aligarh-202002, \newline UP, India}
\email{akhlad6star@gmail.com, akhlad.mm@amu.ac.in}
\begin{document}

\begin{abstract}
As a generalization of geodesic function, in the present paper, we introduce the notion of { \it geodesic $\varphi$-convex function} and deduce some basic properties of $\varphi$-convex function and geodesic $\varphi$-convex function. We  also introduce the concept of { \it geodesic $\varphi$-convex set} and { \it $\varphi$-epigraph} and investigate a characterization of geodesic $\varphi$-convex functions in terms of their $\varphi$-epigraphs. 
\end{abstract}

\maketitle
\section{\textbf{Introduction}}
Convex sets and convex functions play an important role in the study of the theory of nonlinear programming and optimization. But in many situations only convexity is not enough to provide a satisfactory solution of a problem. Hence it is necessary to generalize the concept of convexity notion. Again, due to the curvature and torsion of a Riemannian manifold highly nonlinearity appears in the study of convexity on such a manifold. Geodesics are length minimizing curves,  and the notion of geodesic convex function  arises naturally on a complete Riemannian manifold  and such a concept is investigated recently in   \cite{GS81,UDR94}.\\

\indent 
In 2016 Eshaghi Gordji et. al. \cite{ERD16} defined the notion of the $\varphi$-convex function and deduced Jensen and Hadamard type inequalities for such functions. In the present paper, we have deduced some other proerties of $\varphi$-convex functions. Again, generalizing the concept of  $\varphi$-convex function, we have introduced the notion of {\it geodesic  $\varphi$-convex function} on a complete Riemannian manifold and showed its existence by a proper example (see, Ex: 2.1). Convex sets on a Riemannian manifold have been generalized in different ways such as geodesic $E$-convex function \cite{ASA12}, geodesic semi-$E$-convex function \cite{AAS11}, geodesic semi $E$-$b$-vex functions \cite{AW15} etc. We have also introduced a new class of sets, called, {\it geodesic $\varphi$-convex sets} on a complete Riemannian manifold.\\

\indent 
The paper is organized as follows. Section 2 deals with the rudimentary facts of convex functions and convex sets. Section 3 is devoted to the study of some properties of mean-value like inequalities of  $\varphi$-convex functions. In section 4, we study some properties of geodesic  $\varphi$-convex functions on a complete Riemannian manfold and prove that such notion is invariant under a diffeomorphism. We also study sequentially upper bounded functions on a complete Riemannian manifold and prove that supremum of such  sequence of functions is a geodesic  $\varphi$-convex function. We also obtain a condition for which a geodesic  $\varphi$-convex function has a local minimum. The last section is concerned with $\varphi$-epigraphs on a complete  Riemannian manifold and obtain a characterization of geodesic $\varphi$-convex functions in terms of their $\varphi$-epigraphs (see, Theorem. 5.1.).

\section{\textbf{Preliminaries}}
In this section, we recall some definition and known results of convex and $\varphi$-convex functions and also some results about Riemannian manifolds which will be used throughout the paper. For the detailed discussion of Riemannian manifold we refer \cite{LAN99}.\\
Let $I=[a,b]$ be an interval in $\mathbb{R}$ and $\varphi:\mathbb{R}\times\mathbb{R}\rightarrow\mathbb{R}$ be a bifuction  \cite{ERD16}.
\begin{defn}
A function $f:I\rightarrow\mathbb{R}$ is called convex if for any two points $x,y\in I$ and $t\in[0,1]$
$$f(tx+(1-t)y)\leq tf(x)+(1-t)f(y).$$
\end{defn}
\begin{defn}\cite{ERD16}
A function $f:I\rightarrow\mathbb{R}$ is called $\varphi$-convex if
\begin{equation*}
f(t x+(1-t)y)\leq f(y)+t \varphi(f(x),f(y)), \eqno(1.1)
\end{equation*}
for all $x,y\in I$ and $t\in [0,1]$.
\end{defn}
Especially, if $\varphi (x,y)=x-y$ then $\varphi$-convex function reduces to a convex function.

\begin{defn}\cite{ERD16}
The function $\varphi:\mathbb{R}\times\mathbb{R}\rightarrow\mathbb{R}$ is called \textit{nonnegatively homogeneous} if $\varphi(\lambda x,\lambda y)=\lambda \varphi(x,y)$ for all $x,y\in \mathbb{R}$ and for all $\lambda\geq 0,$ and called \textit{additive} if $\varphi(x_1+x_2,y_1+y_2)=\varphi(x_1,y_1)+\varphi(x_2,y_2)$ for all $x_1,x_2,y_1,y_2\in\mathbb{R}$. If $\varphi$ is both \textit{nonnegatively homogeneous} and \textit{additive} then $\varphi$ is called \textit{nonnegatively linear}.
\end{defn}

Recently, Hanson's \cite{HA81}  generalized convex sets and introduced the concept of invex sets, defined follows.
\begin{defn}\cite{HA81}
A set $K\subseteq \mathbb{R}$ is said to be invex if there exists a function $\eta:\mathbb{R}\times \mathbb{R}\rightarrow \mathbb{R}^n$ such that
$$x, y \in K, \lambda \in [0,1]\Rightarrow y+\lambda \eta(x,y) \in K.$$
\end{defn}

\begin{defn}\cite{ERD16}
Let $K \subseteq R$ be an invex set with respect to $\eta$. A function  $f:K\rightarrow\mathbb{R}$ is said to be $G$-preinvex with respect to  $\eta$ and $\psi$ if 
$$f(y+t \eta(x,y))\leq f(y)+t \psi(f(x),f(y)),$$
for all $x,y \in K$, $t \in [0,1]$.
\end{defn}

Let $(M,g)$ be a complete Riemannian manifold with Riemannian metric $g$ and Levi-Civita connection $\nabla$. We recall that a geodesic is a smooth curve $\alpha$ whose tangent is parallel along the curve $\alpha$, that is, $\alpha$ satisfies the equation $\nabla_{\frac {d\alpha(t)}{dt}} \frac{d\alpha(t)}{dt}=0$. We shall denote the geodesic arc connecting $x \in M$ and $y \in M$ by $\alpha_{xy}:[0,1]\rightarrow M$ such that $\alpha_{xy}(0)=x \in M$, $\alpha_{xy}(1)=y \in M$.
\begin{defn}\cite{UDR94}
A non-empty subset $A$ of $M$ is called totally convex, if it contains every geodesic $\alpha_{xy}$ of $M$, whose end points $x$ and $y$ belong to $A$.
\end{defn}
\begin{defn}\cite{UDR94}
If $A$ is a totally convex set of $M$, then  $f:A\rightarrow\mathbb{R}$ is called geodesically convex if 
$$f(\alpha_{xy}(t))\leq (1-t)f(x)+tf(y)$$ holds for every $x,y \in A$ and $t \in [0,1]$. If the inequality is strict then $f$ is called strictly geodesically convex.
\end{defn}
 Generalizing the notion of $\varphi$-convex function in Riemannian manifold,  we introduce the concept of geodesic $\varphi$-convex function, defined as follows:
\begin{defn}
If $A$ is a totally convex set in $M$,  then  $f:A\rightarrow\mathbb{R}$ is called geodesic $\varphi$-convex if
$$f(\alpha_{xy}(t))\leq f(x)+t\varphi(f(y),f(x)),$$ 
holds for every $x,y\in A$ and $t \in [0,1].$
\end{defn}

{\bf Remark.}
We note that if $f$ is differentiable, then $f$ is called geodesic $\varphi$-convex if and only if $$df_x\dot{\alpha}_{xy}\leq \varphi(f(y),f(x)),$$
where $df_x$ is the differential of $f$ at the point $x \in A \subset M$ and dot denotes the differentiation with respect to $t$.

If the above inequality is strict then $f$ is called strictly geodesic $\varphi$-convex function.

\vspace{.2cm}
\noindent
 Geodesic convex functions are obviously geodesic $\varphi$-convex, where $\varphi(x,y)=x-y$. In the following example we show that geodesic  $\varphi$-convex function on $M$ need not be geodesic convex.

\begin{exmp}
Let $M=\mathbb{R}\times\mathbb{S}^1$ and define $f:M\rightarrow\mathbb{R}$ by $f(x,s)=x^3$. Then $f$ is not geodesic convex in $M$. Now define $\varphi:\mathbb{R}\times\mathbb{R}\rightarrow\mathbb{R}$ by $\varphi(x,y)=x^3-y^3$. We note that for any two points $(x,s_1)$ and $(y,s_2)$ the geodesic joining them is a portion of a helix of the form $\alpha(t)=(tx+(1-t)y,e^{i[t\theta_1+(1-t)\theta_2]})$, for $0\leq t\leq 1$, where $e^{i\theta_1}=s_1$ and $e^{i\theta_2}=s_2$  for some $\theta_1,\theta_2\in[0,2\pi]$. Hence
\begin{eqnarray*}
f(\alpha(t))&=&(tx+(1-t)y)^3\\
&=&t^3(x^3-3x^2y+3xy^2-y^3)+t^2(3x^2y-6xy^2+3y^3)+t(3xy^2-3y^3)+y^3\\
&\leq &y^3+t(x^3-y^3)\\
&=&f(y,s_2)+t\varphi(f(x,s_1),f(y,s_2)).
\end{eqnarray*}
This proves that $f$ is geodesic $\varphi$-convex in $M$.
\end{exmp}

\noindent
Barani et. al. \cite{B07} extended the work of \cite{HA81} to Riemannian manifold and defined geodesic invex sets and geodesic $\eta$-preinvex functions.

\begin{defn}\cite{B07}
Let $M$ be an n-dimensional Riemannian manifold and $\eta:M\times M\rightarrow TM$ be a function such that for every $x, y \in M$, $\eta (x,y) \in T_{y} M$. A non-empty subset $S$ of $M$ is said to be geodesic invex set with respect to $\eta$ if for every $x, y \in S$, there exists a unique geodesic $\alpha_{xy}:[0, 1]\rightarrow M$ such that $$\alpha_{xy}(0)=y,~~ \dot{\alpha}_{xy}(0)=\eta(x, y),~~ \alpha_{xy}(t)\in S,~~~{\forall} ~t\in [0, 1].$$
\end{defn}

\begin{defn}\cite{B07}
Let $S$ be an open subset of $M$ which is geodesic invex with respect to $\eta:M\times M\rightarrow TM$. A function $f:S\rightarrow \mathbb{R}$ is said to be geodesic $\eta$-preinvex if $$f(\alpha_{xy}(t))\leq tf(x)+(1-t)f(y),$$ for every $x,y \in S,~ t\in[0,1]$.
\end{defn}

\noindent
We generalize above definition as follows.

\begin{defn}
Let $S \subseteq M$ be geodesic invex with respect to $\eta$. A function $f:S\rightarrow \mathbb{R}$ is said to be geodesic $\varphi$-preinvex with respect to $\eta$ and $\varphi$, if $$f(\alpha_{xy}(t))\leq f(y)+t \varphi (f(x), f(y)),$$ for every $x,y \in S,~ t\in[0,1]$.
\end{defn}


\section{\textbf{Some properties of $\varphi$-convex functions}}
The notion of $\varphi$-convex functions have been studied in \cite{ERD16},  where Jensen type inequality and Hermite-Hadamard type inequality have been deduced for $\varphi$-convex function. This section deals with some properties of $\varphi$-convex functions. Let $f:I\rightarrow\mathbb{R}$ be a $\varphi$-convex function.
For any two points $x_1$ and $x_2$ in $I$ with $x_1<x_2$, each point $x$ in $(x_1,x_2)$ can be expressed as 
$$x=t x_1+(1-t)x_2 , \text{ where }t=\frac{x_2-x}{x_2-x_1}.$$
 Hence a function $f$ is $\varphi$-convex if 
 \begin{equation*}
 f(x)\leq f(x_2)+\frac{x_2-x}{x_2-x_1}\varphi(f(x_1),f(x_2)),\text{ for }x_1<x<x_2\text{ in }I.
 \end{equation*}
 Rearranging the above terms we get 

$$ \frac{f(x_2)-f(x)}{x_2-x}\geq \frac{\varphi(f(x_1),f(x_2))}{x_1-x_2}\text{ for }x_1<x<x_2\text{ in }I.\eqno(3.1)$$

 So, a function $f$ is $\varphi$-convex if it satisfies the inequality (3.1).
 
 \begin{thm}\label{th1}
If $f:I\rightarrow\mathbb{R}$ is differentiable and $\varphi$-convex in $I$ and $f(x_1)\neq f(x_2)$ then there exist $\xi,\eta\in (x_1,x_2)\subset I$ such that $f'(\xi)\geq \frac{\varphi(f(x_1),f(x_2))}{f(x_1)-f(x_2)}f'(\eta)\geq f'(\eta).$
 \end{thm}
 \begin{proof}
 Since $f$ is $\varphi$-convex, from (2.1)  we get
 \begin{eqnarray}
  \frac{f(x_2)-f(x)}{x_2-x}&\geq& \frac{\varphi(f(x_1),f(x_2))}{x_1-x_2}\text{ for }x\in (x_1,x_2)\subset I,\nonumber \\
  &=&\frac{\varphi(f(x_1),f(x_2))}{f(x_1)-f(x_2)}\frac{f(x_2)-f(x_1)}{x_2-x_1}.\nonumber
  \end{eqnarray}
  Now applying mean-value theorem, the above inequality yields

  \begin{equation*}
 f'(\xi)\geq \frac{\varphi(f(x_1),f(x_2))}{f(x_1)-f(x_2)}f'(\eta) \eqno(3.2)
  \end{equation*}

  for some $\xi\in(x_1,x)\subset(x_1,x_2)$ and $\eta\in (x_1,x_2)$.
  Again by setting $t=1$ in (1.1) we get $\varphi(f(x_1),f(x_2))\geq f(x_1)-f(x_2)$. Hence (3.2) implies that 
 \begin{equation*}
 f'(\xi)\geq \frac{\varphi(f(x_1),f(x_2))}{f(x_1)-f(x_2)}f'(\eta)\geq f'(\eta).\eqno(3.3)
 \end{equation*}
 \end{proof}

 \begin{thm}
 Let $f:I\rightarrow\mathbb{R}$ be a differentiable $\varphi$-convex function. Then for each $x,y,z\in I$ such that $x<y<z$ the following inequality holds:
 $$f'(y)+f'(z)\leq \frac{\varphi(f(x),f(y))+\varphi(f(y),f(z))}{x-z}.$$
 \end{thm}
 \begin{proof}
 Since $f$ is $\varphi$-convex in each interval $[x,y]$ and $[y,z],$ hence
 $$f(tx+(1-t)y)\leq f(y)+t\varphi(f(x),f(y)),$$and
 $$f(ty+(1-t)z)\leq f(z)+t\varphi(f(y),f(z)).$$
 From the above two inequalities, we obtain
  $$\frac{f(tx+(1-t)y)- f(y)+f(ty+(1-t)z)-f(z)}{t}\leq \varphi(f(x),f(y))+\varphi(f(y),f(z)).$$
  Now setting $t\rightarrow 0,$ we have
  $$f'(y)(x-y)+f'(z)(y-z)\leq \varphi(f(x),f(y))+\varphi(f(y),f(z)).$$
 Again $z>y$ implies that $x-z<x-y$ and $x<y$ implies that $x-z<y-z$.
  Thus we get $(x-z)(f'(y)+f'(z))\leq f'(y)(x-y)+f'(z)(y-z).$\\
  Hence we obtain
  $$f'(y)+f'(z)\leq \frac{\varphi(f(x),f(y))+\varphi(f(y),f(z))}{x-z}.$$
 \end{proof}
 
 \section{\textbf{Properties of geodesic $\varphi$-convex functions}}
 \begin{thm}
Let  $A$ be a totally convex set in $M$. Then a function $f:A\rightarrow\mathbb{R}$ is geodesic $\varphi$-convex if and only if for each $x,y\in A$, the function $g_{xy}=f\circ\alpha_{xy}$ is $\varphi$-convex on $[0,1]$.
 \end{thm}
 \begin{proof}
 Let us suppose that $g_{xy}$ is $\varphi$-convex on $[0,1]$. Then for each $t_1,t_2\in [0,1]$,
 $$g_{xy}(st_2+(1-s)t_1)\leq g_{xy}(t_1)+s\varphi(g_{xy}(t_2),g_{xy}(t_1)),\ \forall s\in [0,1].$$
 Now taking $t_1=0$ and $t_2=1$ we get
 $$g_{xy}(s)\leq g_{xy}(0)+s\varphi(g_{xy}(1),g_{xy}(0)),$$
 i.e., $$f(\alpha_{xy}(s))\leq f(x)+s\varphi(f(y),f(x)),\forall x,y\in A \text{ and }\forall s\in [0,1].$$
 Conversely, let $f$ be a $\varphi$-convex function. Now restricting the domain of $\alpha_{xy}$ to $[t_1,t_2]$, we get a geodesic joining $\alpha_{xy}(t_1)$ and $\alpha_{xy}(t_2)$. Now reparametrize this restriction,
 $$\gamma(s)=\alpha_{xy}(st_2+(1-s)t_1),\ s\in [0,1].$$
 Since $f(\gamma(s))\leq f(\gamma(0))+s\varphi(f(\gamma(1)),f(\gamma(0))),$\\
 i.e., $f(\alpha_{xy}(st_2+(1-s)t_1))\leq f(\alpha_{xy}(t_1))+s\varphi(f(\alpha_{xy}(t_2)),f(\alpha_{xy}(t_0))),$\\
 hence $g_{xy}(st_2+(1-s)t_1))\leq g_{xy}(t_1)+s\varphi(g_{xy}(t_2),g_{xy}(t_1)).$\\
 So, $g_{xy}$ is $\varphi$-convex in $[0,1]$.
 \end{proof}

 \begin{thm}
 Let $A\subset M$ be a totally convex set, $f:A\rightarrow\mathbb{R}$ geodesic convex function and $g:I\rightarrow\mathbb{R}$ be non-decreasing $\varphi$-convex function with $Range(f)\subseteq I$. Then $g\circ f:A\rightarrow\mathbb{R}$ is a geodesic $\varphi$-convex function.
 \end{thm}
 \begin{proof}
 Since $f$ is geodesic convex in $A$, hence for any $x,y\in A$ 
 $$f(\alpha_{xy}(t))\leq (1-t)f(x)+tf(y),$$
 where $\alpha_{xy}:[0,1]\rightarrow M$ is a geodesic arc connecting $x$ and $y$.
 Now as $g$ is non-decreasing and $\varphi$-convex, we have
 \begin{eqnarray*}
 g\circ f(\alpha_{xy}(t))
 &\leq& g((1-t)f(x)+tf(y))\\
 &\leq & g\circ f(x)+t\varphi(g\circ f(y),g\circ f(x)).
 \end{eqnarray*}
 Hence $g\circ f$ is geodesic $\varphi$-convex in $A$.
 \end{proof}

 {\bf Remark.}  In Theorem 4.2, if $g$ is strictly $\varphi$-convex, then $g\circ f$ is also strictly geodesic $\varphi$-convex.

 \begin{thm}
 Let $f_i:A\rightarrow\mathbb{R}$ be geodesic $\varphi$-convex functions for $i=1,2,\cdots,n$ and $\varphi$ be nonnegatively linear. Then for $\lambda_i\geq 0,\ i=1,2,\cdots,n$, the function $f=\sum_{i=1}^{n}\lambda_if_i:A\rightarrow\mathbb{R}$ is geodesic $\varphi$-convex.
 \end{thm}
 \begin{proof}
 Let $x,y\in A$ then 
 \begin{eqnarray*}
 f(\alpha_{xy})&=&\sum_{i=1}^{n}\lambda_if_i(\alpha_{xy})\\
 &\leq & \sum_{i=1}^{n}\lambda_i[f_i(x)+t \varphi(f(y),f(x))]\\
 &= & \sum_{i=1}^{n}[\lambda_if_i(x)+t \varphi(\lambda_if(y),\lambda_if(x))]\\
 &=& f(x)+t\varphi(f(y),f(x)).
 \end{eqnarray*}
 Hence $f$ is geodesic $\varphi$-convex in $A$.
 \end{proof}

  Let $M$ and $N$ be two complete Riemannian manifolds and $\nabla$ be the Levi-Civita connection of $M$. If $F:M\rightarrow N$ be a diffeomorphism, then  $F_*\nabla=\nabla_1$ is an affine connection of $N$. Hence if $\gamma$ is a geodesic in $(M,\nabla)$, then $F\circ\gamma$ is also a geodesic in $(N,\nabla_1)$ \cite[pp. 66]{UDR94}.
  \begin{thm}
  Let $f:A\rightarrow\mathbb{R}$ be a geodesic $\varphi$-convex function. If $F:(M,\nabla)\rightarrow (N,\nabla_1)$ is a differmorphism then $f\circ F^{-1}:F(A)\rightarrow\mathbb{R}$ is a geodesic $\varphi$-convex function, where $\nabla_1=F_*\nabla$.
  \end{thm}
 \begin{proof}
 Let $x,y\in A$ and $\alpha_{xy}$ be a geodesic arc joining $x$ and $y$. Since $F$ is a diffeomorphism, hence $F(A)$  is totally geodesic \cite{UDR94} and $F\circ\alpha_{xy}$  is a geodesic joining $F(x)$ and $F(y)$. Now we have
 \begin{eqnarray*}
 (f\circ F^{-1})(F(\alpha_{xy}))&=& f(\alpha_{xy})\\
 &\leq & f(x)+t\varphi(f(y),f(x))\\
 &=& (f\circ F^{-1})(F(x))+t\varphi((f\circ F^{-1})(F(y)),(f\circ F^{-1})(F(x))))
 \end{eqnarray*}
  i,e., $f\circ F^{-1}$ is geodesic $\varphi$-convex on $F(A)$.
 \end{proof}

 \begin{thm}\label{thm3}
  Let $f:B\rightarrow\mathbb{R}$ be a geodesic $\varphi$-convex function and $\varphi$ be bounded from above on $f(B)\times f(B)$ with an upper bound $M_\varphi$, where $B$ is a convex set in $\mathbb{R}^n$ with $Int(B)\neq \phi$. Then $f$ is continuous on $Int(B)$.
 \end{thm}

  \begin{proof}
  Let $a\in Int(B)$. Then there exists an open ball $B(a,h)\subset Int(B)$ for some $h>0$. Now choose $r$ $(0<r<h)$ such that the closed ball $\bar{B}(a,r+\epsilon)\subset B(a,h)$ for some arbitrary small $\epsilon>0$. Choose any $x,y\in \bar{B}(a,r)$. Set $z=y+\frac{\epsilon}{\norm{y-x}}(y-x)$ and $t=\frac{\norm{y-x}}{\epsilon+\norm{y-x}}$. Then it is obvious that $z\in \bar{B}(a,r+\epsilon)$ and $y=tz+(1-t)x$. Thus $f(y)\leq f(x)+t\varphi(f(z),f(x))\leq f(x)+tM_\varphi$. Hence we get
  $$f(y)-f(x)\leq tM_\varphi\leq \frac{\norm{y-x}}{\epsilon}M_\varphi=K\norm{y-x},$$
  where $K=M_\varphi/\epsilon$. And if we interchange the position of $x$ and $y$, then we also get $f(x)-f(y)\leq K\norm{y-x}$. Hence $|f(x)-f(y)|\leq K\norm{y-x}$. Since $\bar{B}(a,r)$ is arbitrary hence $f$ is continuous on $Int(B)$. 
  \end{proof}
\begin{defn}
A bifunction $\varphi:\mathbb{R}^2\rightarrow\mathbb{R}$ is called \textit{sequentially upper bounded} if
$$\sup_n\varphi(x_n,y_n)\leq \varphi(\sup_nx_n,\sup_ny_n)$$
for any two bounded real sequences $\{x_n\}$ and $\{y_n\}$.
\end{defn}
\begin{exmp}
The functions $\varphi(x,y)=x+y$ and $\psi(x,y)=xy, \ \forall (x,y)\in\mathbb{R}^2$ are sequentially upper bounded functions.
\end{exmp}
\begin{prop}
Let $A\subseteq M$ be totally convex set and $\{f_n\}_{n\in \mathbb{N}}$ be a non-empty family of geodesic $\varphi$-convex functions on $A$, where $\varphi$ is sequentially upper bounded. If $\sup_nf_n(x)$ exists for each $x\in A$ then $f(x)=\sup_nf_n(x)$ is also a geodesic $\varphi$-convex function.
\end{prop}
\begin{proof}
Let $x,y\in A$ and $\alpha_{xy}:[0,1]\rightarrow M$ be a geodesic connecting $x$ and $y$. Then 
\begin{eqnarray*}
f(\alpha_{xy}(t))&=&\sup_nf_n(\alpha_{xy}(t))\\
&\leq & \sup_n\{f_n(x)+t\varphi(f_n(y),f_n(x))\\
&\leq & \sup_nf_n(x)+t\sup_n\varphi(f_n(y),f_n(x))\\
&\leq & \sup_nf_n(x)+t\varphi(\sup_n f_n(y),\sup_n f_n(x))\\
&=& f(x)+t\varphi(f(y),f(x))
\end{eqnarray*}
Hence $f$ is a geodesic $\varphi$-convex function on $A$.
\end{proof}
\begin{thm}
Let $f:A\rightarrow\mathbb{R}$ be a geodesic $\varphi$-convex function. If $f$ has local minimum at $x_0\in Int(A),$ then $\varphi(f(x),f(x_0))\geq 0$ for all $x\in A$.
\end{thm}
\begin{proof}
Since $x_0\in Int(A)$, $B(x_0,r)\subset A$ for some $r>0$. Take any point $x\in A$. Then there exists a geodesic $\alpha_{x_0 x}:[0,1]\rightarrow M$ belonging to $A$ and $f(\alpha_{x_0x}(t))\leq f(x_0)+t\varphi(f(x),f(x_0))$. Now there exists $\xi$ such that $0<\xi\leq 1$ and $\alpha_{x_0x}(t)\in B(x_0,r), \ \forall t\in [0,\xi].$\\
As $f$ has local minimum at $x_0$, we obtain
$$f(x_0)\leq f(\alpha_{x_0x}(\xi))\leq f(x_0)+\xi\varphi(f(x),f(x_0)).$$
Thus,
$$\varphi(f(x),f(x_0))\geq 0,\ \forall x\in A.$$
\end{proof}
\begin{thm}
Let $f:A\rightarrow\mathbb{R}$ be geodesic $\varphi$-convex and $\varphi$ be bounded from above on $f(A)\times f(A)$ with an upper bound $M_\varphi$. Then $f$ is continuous on $Int(A)$. 
\end{thm}
\begin{proof}
Let $a\in Int(A)$ and $(U,\psi)$ be a chart containing $a$. Since $\psi$ is a differmorphism so using Theorem \ref{thm3}, $f\circ\psi^{-1}:\psi(U\cap Int(A))\rightarrow\mathbb{R}$ is also $\varphi$-convex and hence continuous. Therefore, we get $f=f\circ\psi^{-1}\circ\psi:U\cap Int(A)\rightarrow\mathbb{R}$ is continuous. Since $a$ is arbitrary, $f$ is continuous on $Int(A)$.
\end{proof}
\begin{prop}
Let $\{\varphi_n:n\in\mathbb{N}\}$ be a collection of bifunctions such that $f:A\rightarrow\mathbb{R}$ is geodesic $\varphi_n$-convex functions for each $n$. If $\varphi_n\rightarrow\varphi$ as $n\rightarrow\infty$ then $f$ is also a geodesic $\varphi$-convex function.
\end{prop}

\begin{prop}
Let $\{\varphi_n:n\in\mathbb{N}\}$ be a collection of bifunctions such that $f:A\rightarrow\mathbb{R}$ is geodesic $\sum_{i=1}^{n}\varphi_i$-convex function for each $n$. If $\sum_{n\in\mathbb{N}}\varphi_n$ converges to $\varphi$ then $f$ is also a geodesic $\varphi$-convex function.
\end{prop}

\begin{thm}
If $f:A\rightarrow\mathbb{R}$ is strictly geodesic $\varphi$-convex with $\varphi$ as antisymmetric function, then $df_x\dot{\alpha}_{xy}\neq df_y\dot{\alpha}_{xy}$ for any $x,y\in A$, $x\neq y$.
\end{thm}
\begin{proof}
Let $\alpha_{xy}:[0,1]\rightarrow M$ be a geodesic with starting point $x$ and ending point $y$. \\
Now define $\alpha_{yx}(t)=\alpha_{xy}(1-t),\ t\in [0,1]$. Then $\alpha_{yx}$ is a geodesic with starting point $y$ and ending point $x$ and $df_y\dot{\alpha}_{yx}=-df_y\dot{\alpha}_{xy}$. \\
On contrary, suppose that $df_x\dot{\alpha}_{xy}= df_y\dot{\alpha}_{xy}$. Now since $f$ is geodesic $\varphi$-convex,  we get
\begin{equation*}
df_x\dot{\alpha}_{xy}<\varphi(f(y),f(x))
\end{equation*}
 and 
\begin{equation*}
df_y\dot{\alpha}_{yx}<\varphi(f(x),f(y)).
\end{equation*}
Since $ df_y\dot{\alpha}_{yx}=-df_y\dot{\alpha}_{xy}$, we get
\begin{eqnarray*}
-df_y\dot{\alpha}_{xy}&<&\varphi(f(x),f(y)). \\
\end{eqnarray*}
Using antisymmetry property of $\varphi$, we obtain
\begin{eqnarray*}
df_y\dot{\alpha}_{xy}&>&\varphi(f(y),f(x)).\\
\end{eqnarray*}
Since $df_x\dot{\alpha}_{xy}= df_y\dot{\alpha}_{xy}$, we get
\begin{eqnarray*}
df_x\dot{\alpha}_{xy}&>&\varphi(f(y),f(x)).\\
\end{eqnarray*}
Hence $\varphi(f(y),f(x))<\varphi(f(y),f(x))$, which is a contradiction.  Hence $df_x\dot{\alpha}_{xy}\neq df_y\dot{\alpha}_{xy}$.
\end{proof}

\begin{thm}
Let $S \subseteq M$ be a geodesic invex set with respect to $\eta$. Suppose that  $f:S\rightarrow \mathbb R$ is geodesic $\varphi$-convex and $g:I\rightarrow \mathbb R$ is a non-decreasing $G$-preinvex function with respect to $\phi$ and $\psi$ such that range $(f) \subset I$. Then $gof:I\rightarrow \mathbb R$ is a geodesic $\psi$-convex function. 
\end{thm}

\begin{proof}
Let $x,y \in S$. Since $f$ is a geodesic $\varphi$-convex function, we have
$$f(\alpha_{xy}(t))\leq f(y)+t \varphi (f(x), f(y)).$$
Since $g$ is non-decreasing G-preinvex function, we have
\begin{eqnarray*}
g(f(\alpha_{xy}(t))) &\leq&  g(f(y)+t \varphi (f(x), f(y)))\\
&\leq& g(f(y))+t \psi (g(f(x)), g(f(y))).
\end{eqnarray*}
\end{proof}

\section{\textbf{$\varphi$-Epigraphs}}
In this section we have introduced the notion of $\varphi$-epigraphs on complete Riemannian manifolds and obtained a characterization of geodesic $\varphi$-convex function in terms of their $\varphi$-epigraphs. 
\begin{defn}
A set $B\subseteq M\times\mathbb{R}$ is said to be geodesic $\varphi$-convex if and only if for any two points $(x,\alpha),(y,\beta)\in B$ imply
$$(\alpha_{xy}(t),\alpha+t\varphi(\beta,\alpha))\in B, \quad 0\leq t\leq 1.$$
\end{defn}
If $A$ is a totally geodesic convex subset  of $M$ and $K$ is an invex subset of $\mathbb{R}$ with respect to $\varphi$,  then $A\times K$ is geodesic $\varphi$-convex.\\
\indent Now the $\varphi$-epigraph of $f$ is defined by
$$E_\varphi(f)=\{(x,\alpha)\in M\times\mathbb{R}:f(x)\leq \alpha\}.$$
\begin{defn}\cite{LIN04}
A function $\varphi:\mathbb{R}\times\mathbb{R}\rightarrow\mathbb{R}$ is \textit{non-decreasing} if $x_1\leq x_2$ and $y_1\leq y_2$ implies $\varphi(x_1,y_1)\leq \varphi(x_2,y_2)$ for $x_1,x_2,y_1,y_2\in\mathbb{R}$.
\end{defn}

\begin{thm}\label{th4}
Let $A$ be a totally geodesic convex subset of $M$ and $\varphi$ be non-decreasing. Then $f:A\rightarrow\mathbb{R}$ is geodesic $\varphi$-convex if and only if $E_\varphi(f)$ is a geodesic $\varphi$-convex set.
\end{thm}
\begin{proof}
Suppose that $f:A\rightarrow\mathbb{R}$ is a geodesic $\varphi$-convex. Let $(x,\alpha),(y,\beta)\in E_\varphi(f)$. Then $f(x)\leq \alpha$ and $f(y)\leq \beta$. Since $f$ is geodesic $\varphi$-convex, hence 
$$f(\alpha_{xy}(t))\leq f(x)+t\varphi(f(y),f(x)),$$
for any geodesic $\alpha_{xy}:[0,1]\rightarrow M$ connecting $x$ and $y$.
Since $\varphi$ is non-decreasing, so 
$$f(\alpha_{xy}(t))\leq \alpha+t\varphi(\beta,\alpha),\ \forall t\in [0,1].$$
Hence $$(\alpha_{xy}(t),\alpha+t\varphi(\beta,\alpha))\in E_\varphi(f)\ \forall t\in[0,1].$$ That is, $E_\varphi(f)$ is a geodesic $\varphi$-convex set.\\
\indent Conversely, assume that $E_\varphi(f)$ is a geodesic $\varphi$-convex set.\\
Let $x,y\in A$. Then $(x,f(x)),(y,f(y))\in E_\varphi(f)$. Hence for any $t\in[0,1],$
$$(\alpha_{xy}(t),f(x)+t\varphi(f(y),f(x)))\in E_\varphi(f),$$
which implies that $$f(\alpha_{xy}(t))\leq f(x)+t\varphi(f(y),f(x)),$$
for any $x,y\in A$ and for any geodesic $\alpha_{xy}:[0,1]\rightarrow M$. So, $f$ is a geodesic $\varphi$-convex function.
\end{proof}
\begin{thm}\label{th5}
Let $A_i,\ i\in \Lambda$, be a family of geodesic $\varphi$-convex sets. Then their intersection $A=\bigcap_{i\in\Lambda}A_i$ is also a geodesic $\varphi$-convex set.
\end{thm}
\begin{proof}
Let $(x,\alpha),(y,\beta)\in A$. Then for each $i\in \Lambda,\ (x,\alpha),(y,\beta)\in A_i$. Since each $A_i$ is geodesic $\varphi$-convex for each $i\in \Lambda$, hence
$$(\alpha_{xy}(t),\alpha+t\varphi(\beta,\alpha))\in A_i,\quad \forall t\in [0,1].$$
This implies
$$(\alpha_{xy}(t),\alpha+t\varphi(\beta,\alpha))\in \bigcap_{i\in\Lambda}A_i,\quad \forall t\in [0,1].$$
Hence $A$ is a geodesic $\varphi$-convex set.
\end{proof}
As a direct consequence of Theorem \ref{th4} and Theorem \ref{th5},  we get the following corollary.
\begin{cor}
Let $\{f_i\}_{i\in\Lambda}$ be a family of geodesic $\varphi$-convex functions defined on a totally convex set $A\subseteq M$ which are bounded above and $\varphi$ be non-decreasing. If the $\varphi$-epigraphs $E_\varphi(f_i)$ are geodesic $\varphi$-convex sets, then $f=\sup_{i\in\Lambda}f_i$ is also a geodesic $\varphi$-convex function on $A$.
\end{cor}

\end{document}